\documentclass[11pt]{amsart}

\usepackage{lscape}

\newcounter{cnt1}
\newcounter{cnt2}
\newcommand{\blr}{\begin{list}{$(\roman{cnt1})$}
    {\usecounter{cnt1} \setlength{\topsep}{0pt}
        \setlength{\itemsep}{0pt}}}
\newcommand{\bla}{\begin{list}{$($\alph{cnt2}$)$}
    {\usecounter{cnt2} \setlength{\topsep}{0pt}
        \setlength{\itemsep}{0pt}}}
\newcommand{\el}{\end{list}}
\newtheorem{thm}{Theorem}

{\newtheorem{Def}[thm]{Definition}
\newtheorem{prop}[thm]{Proposition}
\newtheorem{rem}[thm]{Remark}
\newcommand{\Rem}{\begin{rem} \rm}
\newcommand{\bdfn}{\begin{Def} \rm}
\newcommand{\edfn}{\end{Def}}

\begin{document}
\large
\title[Computing subdifferential limits]{Computing subdifferential limits of operators on Banach spaces}
\author[Rao]{T. S. S. R. K. Rao}
\address[T. S. S. R. K. Rao]
{Department of Mathematics\\
Shiv Nadar University \\
Delhi (NCR) \\ India,
\textit{E-mail~:}
\textit{srin@fulbrightmail.org}}
\begin{abstract}
Let $X,Y$ be real, infinite dimensional, Banach spaces. Let ${\mathcal L}(X,Y)$ be the space of bounded operators . An important aspect of understanding differentiability of the operator norm at $A \in{\mathcal L}(X,Y)$ is to estimate the limit $lim_{t \rightarrow 0^+} \frac{\|A+tB\|-\|A\|}{t}$, for $B \in {\mathcal L}(X,Y)$ (which always exists) using the values of $B$ on the state space $S_A= \{\tau \in {\mathcal L}(X,Y)^\ast: \tau(A)=\|A\|~,~\|\tau\|=1\}$. In this paper we give several examples of Banach spaces, including the $\ell^p$ spaces (for $1 <p<\infty$) where this can be either explicitly evaluated or an easier estimate is possible. We use the notion of norm-weak upper-semi-continuity (usc, for short) of the preduality map to achieve this. Our results also show that the operator subdifferential limit is related to the
corresponding subdifferential limit of the vectors in the range space, when $A^{\ast\ast}$ attains its norm.
\end{abstract}
\maketitle
\section { Introduction}
For a real Banach space $X$, let $X_1$ denote the closed unit ball and $\partial_e X_1$ denote the set of extreme points (when non-empty). For $ 0 \neq x \in X$, let $S_x$ denote the state space $\{x^\ast \in X^\ast_1: x^\ast(x) = \|x\|\}$. This is a weak$^\ast$-compact extreme convex subset of $X^\ast_1$ and it is well known that
$$\lim_{t \rightarrow 0^+}\frac{\|x+ty\|-\|x\|}{t} = sup\{x^\ast(y): x^\ast \in S_x\}.$$
Since $y$ acts as a weak$^\ast$-continuous affine function on $S_x$ by an application of the Krein-Milman theorem, one knows that the supremum above is equal to $x^\ast(y)$ for some
$x^\ast \in \partial_e X^\ast_1$ such that $x^\ast(x)=\|x\|$, see Proposition 2.24 in \cite{P}.
\vskip 1em
In this paper we are interested in studying the subdifferential limit as above (more precisely the righthand directional derivative) in ${\mathcal L}(X,Y)$ for infinite dimensional Banach spaces $X,Y$. See \cite{W}, \cite{CP}, \cite{TW}, \cite{MR} for differentiability of the norm in $C^\ast$-algebras and their weaker forms, also the recent work \cite{S}, \cite{S1}
for the structure of the state space in the case of operators on Hilbert spaces and other $C^\ast$-algebras. An analysis of the structure of the state space, in spaces of operators was also carried out in \cite{R2}.
\vskip 1em
For non-zero vectors $x^{\ast\ast} \in X^{\ast\ast}_1$, $y^\ast \in Y^\ast_1$, we call the functional $x^{\ast\ast}\otimes y^\ast$ defined for $T \in {\mathcal L}(X,Y)$, by $(x^{\ast\ast}\otimes y^\ast)(T) = x^{\ast\ast}(T^\ast(y^\ast))$, an elementary functional. Let ${\mathcal K}(X,Y)$ denote the space of compact operators. We note that ${\mathcal L}(X,Y)^\ast_1 = \overline{CO}\{x^{\ast\ast}\otimes y^\ast: x^{\ast\ast} \in X^{\ast\ast}_1~,~y^\ast \in Y^\ast_1\}$ where the closure of the convex hull is taken in the weak$^\ast$-topology and similar statement holds for ${\mathcal K}(X,Y)^\ast_1$ (this is because such functionals determine the norm of an operator).
\vskip 1em
We first note that for $A,B \in {\mathcal L}(X,Y)$ when both are compact, the subdifferential limit is attained at an elementary functional. We next show that in the general case, if the limit is attained at an elementary functional $x^{\ast\ast} \otimes y^\ast$, then $lim_{t \rightarrow 0^+} \frac {\|A^\ast(y^\ast)+tB^\ast(y^\ast)\|}{t}= x^{\ast\ast}(B^\ast(y^\ast))$, leading to the conclusion, if $A$ is a smooth operator, then the point-wise differential limit is attained at a unique functional. This is another illustration of how local behaviour can be extracted from global differentiability.
\vskip 1em
We consider a Banach space $X$ as canonically embedded in its bidual, $X^{\ast\ast}$. For $0 \neq x \in X$, let $S_x = \{x^\ast \in X^\ast_1: x^\ast(x) = \|x\|\}$ and $S^x = \{\tau \in X^{\ast\ast\ast}_1:\tau(x) = \|x\|\}$.
\vskip 1em
\emph{This notation is specific to bidual spaces and in all other case the state space has the same definition as in the early part of the Introduction}.
In what follows we are interested in the question:
\vskip 1em
 when is $lim_{t \rightarrow 0+}\frac{\|x+t\tau\|-\|x\|}{t}  = sup \{\tau(f): f \in S_x\}$ for all $\tau \in X^{\ast\ast}$?
We know that $lim_{t \rightarrow 0+}\frac{\|x+t\tau\|-\|x\|}{t} = sup \{\tau(f): f \in S^x\}$ for all $\tau \in X^{\ast\ast}$.
\vskip 1em
Thus we need equality of the suprema for all $\tau \in X^{\ast\ast}$. This is clearly seen to be equivalent to  $S_x$ being weak$^\ast$-dense in $S^x$.
\vskip 1em
We next recall that for $x^\ast \in X^\ast$ that attains its norm at a $x \in X_1$, the map $\rho(x^\ast) = S_x$ is called the preduality map. We note that points of $X^\ast$ that attain the norm, is norm dense in $X^\ast$.
\begin{Def}
We recall from \cite{GI} that $x^\ast$ is said to be a point of norm weak usc for $\rho$, if given a weak neighbourhood $V$ of $0$ in $X$, there is a $\delta >0$ such that for a norm attaining $x^\ast_1$ with $\|x^\ast -x^\ast_1\|< \delta$ implies $\rho(x^\ast_1) \subset V
+ \rho(x^\ast)$.
 \end{Def}
Thus by Lemma 2.2 in \cite{GI} the denseness of the state space is equivalent to $x$ being a point of  norm-weak upper-semi-continuity of the preduality map $\rho(x) = S_x$ on $X^{\ast\ast}$. More generally, the same conclusions apply for any point of norm-weak usc in a dual space $X^\ast$ (it is assumed that the functional attains its norm).
\vskip 1em
Since in many classical situations, like in $\ell^p$ spaces ($1 <p<\infty$) or when $X,Y$ are reflexive with the metric approximation property, one has the inclusion map on the space of compact operators, ${\mathcal K}(X,Y) \subset {\mathcal L}(X,Y)$ is the canonical embedding of ${\mathcal K}(X,Y)$ in its bidual ${\mathcal L}(X,Y)$ (see the discussion on page 247 in \cite{DU}), we take the
help of norm-weak usc operators in ${\mathcal K}(X,Y)$ or ${\mathcal L}(X,Y)$ to evaluate the limits. We refer to the monograph \cite{DU},  for basic theory of tensor products and structure of spaces of operators.
\vskip 1em
We show that when the projective tensor product space $X\otimes_{\pi}Y$ has the Radon-Nikodym property (RNP), then for $A \in {\mathcal L}(X,Y^\ast)$, $\lim_{t \rightarrow0^+} \frac{\|A+tB\|-\|A\|}{t}= sup \{x^{\ast\ast}(B^\ast(y^\ast)): x^{\ast\ast} \otimes y^\ast \in S_A\}$, when $A$ is a point of norm-weak usc for the preduality map on ${\mathcal L}(X,Y^\ast) = (X \otimes_{\pi}Y)^\ast$.
\vskip 1em
We also present results on the behaviour of norm-weak usc points in $M$-ideals (see \cite{HWW} Chapter I) and apply them in the context of the Calkin space ${\mathcal L}(X,Y)/{\mathcal K}(X,Y)$. We conclude by considering the case when subdifferentiable limits exists uniformly over the unit ball (the so called SSD points, see \cite{FP}).
\vskip 1em
Most results go through in the complex case also with real part of functional replacing the complex functional. We refer to the monograph \cite{DU} for the geometry of spaces of vector-valued functions and spaces of operators.
\section{Main results}
We recall that $\partial_e {\mathcal K}(X,Y)^\ast_1
= \{x^{\ast\ast}\otimes y^\ast: x^{\ast\ast} \in \partial_e X^{\ast\ast}_1~,~y^\ast \in \partial_e Y^\ast_1\}$, see Theorem VI.1.3 in \cite{HWW}. Thus for $A,B \in {\mathcal K}(X,Y)$, by our remarks in the Introduction, we get:
$$lim_{t \rightarrow 0^+}\frac{\|A+tB\|-\|A\|}{t} = (x^{\ast\ast}\otimes y^\ast)(B) = x^{\ast\ast}(y^\ast(B))$$
for some $x^{\ast\ast} \in \partial_e X^{\ast\ast}_1$, $y^\ast \in \partial_e Y^\ast_1$ such that
$x^{\ast\ast}(A^\ast(y^\ast))= \|A\|$.
\vskip 1em
We thus have, $\|A\|= \|A^\ast\| =\|A^{\ast\ast}\| \leq \|A^{\ast\ast}(x^{\ast\ast})\| \leq \|A^{\ast\ast}\|$ as well as $\|A^\ast\| \leq \|A^\ast(y^\ast)\| \leq \|A^\ast\|=\|A\|$. So that $\|A^{\ast\ast}(x^{\ast\ast})\| = \|A\|$ and $\|A^\ast(y^\ast)\|=\|A\|$.
\vskip 1em
\begin{prop}
	For operators $A,B \in {\mathcal L}(X,Y)$ suppose $$lim_{t \rightarrow 0^+}\frac{\|A-tB\|-\|A\|}{t} = x^{\ast\ast}(B^\ast(y^\ast))$$ for unit vectors $x^{\ast\ast} \in X^{\ast\ast}$, $y^\ast \in Y^\ast$ such that $x^{\ast\ast}(A^\ast(y^\ast))=\|A\|$. Then $lim_{t \rightarrow0^+}\frac {\|A^\ast(y^\ast)+tB^\ast(y^\ast)\|}{t} = x^{\ast\ast}(B^\ast)(y^\ast)$. If further $A$ is a smooth point, then the point-wise differential limit is attained at a unique functional of the state space $S_{A^\ast(y^\ast)}$. 	
\end{prop}
\begin{proof}
	Let $t>0$,  $tx^{\ast\ast}(B^\ast(y^\ast)) = x^{\ast\ast}(A^\ast(y^\ast)+tB^\ast(y^\ast))-x^{\ast\ast}(A^\ast(y^\ast)) \leq \|A^\ast(y^\ast)+tB(y^\ast)\|-\|A^\ast(y^\ast)\|= \|A^\ast(y^\ast)+tB(y^\ast)\|-\|A^\ast\|$. Now dividing by $t$ and taking the limits,
$$x^{\ast\ast}(B^\ast)(y^\ast) \leq lim_{t \rightarrow 0^+} \frac{\|A^\ast(y^\ast)+tB^\ast(y^\ast)\|-\|A^\ast(y^\ast)\|}{t}$$
$$ \leq lim_{t \rightarrow 0^+} \frac{\|A^\ast + t B^\ast\|-\|A^\ast\|}{t} = x^{\ast\ast}(B^\ast(y^\ast)).$$
 Thus equality holds every where, giving, $$lim_{t \rightarrow 0^+}\frac{\|A^\ast(y^\ast)+tB^\ast(y^\ast)\|}{t}= x^{\ast\ast}(B^\ast(y^\ast)).$$
\vskip 1em
Suppose $A$ is a smooth operator and suppose $x^{\ast\ast}(B^\ast(y^\ast))= x_0^{\ast\ast}(B^\ast(y^\ast))$ for some $x_0^{\ast\ast}$ such that $x_0^{\ast\ast}(A^\ast(y^\ast))=\|A^\ast(y^\ast)\|$. We have $(x^{\ast\ast}_0 \otimes y^\ast)(A)= x_0^{\ast\ast}(A^\ast(y^\ast))=\|A^\ast(y^\ast)\| \leq \|A^\ast\|= \|A\| = (x^{\ast\ast} \otimes y^\ast) (A)$. Since $A$ is a smooth point, we get $x^{\ast\ast} \otimes y^\ast = x^{\ast\ast}_0 \otimes y^\ast$ . Since $\|y^\ast\|=1 $, we get $x^{\ast\ast} = x^{\ast\ast}_0$.
\end{proof}
\begin{rem}
	The hypothesis clearly forces $A^\ast$ to attain its norm. We recall from  \cite{Z} that the set of operators whose adjoint attains its norm is a dense set in ${\mathcal L}(X,Y)$. In one of the interesting classes that we will be discussing below, ${\mathcal K}(X,Y)$ is a $M$-ideal in ${\mathcal L}(X,Y)$ and in this case by Proposition VI.4.8 (b), operators whose adjoint fails to attain the norm is a nowhere dense set.
\end{rem}
We next consider this problem in ${\mathcal L}(X,Y^\ast)$, which is the dual of the projective tensor product space $X \otimes_{\pi} Y$.
Since we are in a dual space, for $A \in {\mathcal L}(X,Y^\ast)$ that attains its norm, we use the notation $S_A$ and $S^A$ for state spaces in the predual and the dual. The hypothesis assumed in the following theorem is satisfied when $X,Y$ are reflexive spaces, see \cite{DU} Chapter VIII. We recall that in any von Neumann algebra, the identity or any unitary,  is a point of norm-weak u.s.c for the preduality map, as well as $I \in {\mathcal L}(X^\ast)$ for any dual space $X^{\ast}$.
\begin{thm} Suppose $X \otimes_{\pi}Y$ has the Radon-Nikodym property. Let $A \in {\mathcal L}(X,Y^\ast)$ be a point of norm-weak u.s.c for the preduality map on ${\mathcal L}(X,Y^\ast)$. Then for any $B \in {\mathcal L}(X,Y^\ast)$,
$lim_{t \rightarrow 0^+}\frac{\|A+tB\|-\|A\|}{t} =sup \{x^{\ast\ast}(B^\ast(y^\ast)):  x^{\ast\ast}\otimes y^\ast \in S_A\}$.
\end{thm}
\begin{proof} Let $B \in {\mathcal L}(X,Y^\ast)$. We have $lim_{t \rightarrow 0^+}\frac{\|A+tB\|-\|A\|}{t} =sup \{\tau(B):  \tau \in S^A\}$
and since $A$ is a point of norm-weak u.s.c for the preduality map, this limit is also equal to $sup\{\tau(B): \tau \in S_A\}$. We note that $S_A$
is a norm closed extreme convex subset of $(X \otimes_{\pi}Y)_1$. Since $X \otimes_{\pi} Y$ has the RNP, by a theorem of Phelps ( see \cite{DU}, Theorem 3, page 202), it is the norm closed convex hull of its strongly exposed and hence extreme points. As $S_A$ is an extremal set, these extreme points are also extreme in $X \otimes_{\pi}Y$ and therefore are elementary  functionals on ${\mathcal L}(X,Y^\ast)$. Since $B$ is norm continuous on $S_A$, we have $sup\{\tau(B): \tau \in S_A\} = sup \{x^{\ast\ast}(B^\ast(y^\ast)): x^{\ast\ast}\otimes y^\ast \in \partial_e S_A\}$.
\end{proof}
\vskip 1em
We recall from \cite{HWW} Chapter I, that a closed subspace $J \subset X$ is said to be a $M$-ideal, if there is a linear projection $P: X^\ast \rightarrow X^\ast$ such that $ker(P) = J^\bot$ and $\|P(x^\ast)\|+\|x^\ast-P(x^\ast)\| = \|x^\ast\|$, for all $x^\ast \in X^\ast$ (called $L$-projection). Our applications are motivated by the existence of large classes of Banach spaces $X,Y$, where ${\mathcal K}(X,Y)$ is a $M$-ideal in ${\mathcal L}(X,Y)$, which in some cases is the bidual of ${\mathcal K}(X,Y)$. See Chapter III and VI of \cite{HWW}. Under this hypothesis, we also study the Calkin space ${\mathcal L}(X,Y)/{\mathcal K}(X,Y)$.
\vskip 1em
\begin{thm}
Let $X,Y$ be Banach spaces such that ${\mathcal K}(X,Y)$ is a $M$-ideal in ${\mathcal L}(X,Y)$. Let $A \in {\mathcal L}(X,Y)$ be such that $d(A,{\mathcal K}(X,Y))< \|A\|$. Then for any $B \in {\mathcal L}(X,Y)$, $lim_{t \rightarrow 0^+}\frac{\|A+tB\|-\|A\|}{t} = x^{\ast\ast}(B^\ast(y^\ast))$ for some $(x^{\ast\ast} \otimes y^\ast) \in \partial_e S_A$ and hence $\lim_{t \rightarrow 0^+}\frac{\|A^\ast(y^\ast)+B^\ast(y^\ast\|}{t} = x^{\ast\ast}(B^\ast(y^\ast))$.
\end{thm}
\begin{proof}
  We use the identification of $\partial_e{\mathcal K}(X,Y)^\ast_1$ given earlier. Since ${\mathcal K}(X,Y)$ is a $M$-ideal in ${\mathcal L}(X,Y)$, we have $${\mathcal L}(X,Y)^\ast = {\mathcal K}(X,Y)^\ast \bigoplus_1 {\mathcal K}(X,Y)^\bot$$ ($\ell^1$-direct sum).
  Let $\pi: {\mathcal L}(X,Y) \rightarrow {\mathcal L}(X,Y)/{\mathcal K}(X,Y)$ be the quotient map. Let $\tau \in  \partial_e {\mathcal L}(X,Y)^\ast_1$ be such that $\|A\|= \tau(A)$. We recall that these are precisely extreme points of $S_A$. Because of the $\ell^1$-decomposition of the dual space, $\tau \in \partial_e{\mathcal K}(X,Y)^\ast_1 \cup \partial_e ({\mathcal K}(X,Y)^\bot_1)$. Since $\|\pi(A)\|< \|A\|$, we get
  $\tau \in \partial_e {\mathcal K}(X,Y)^\ast_1$.
  \vskip 1em
  We already know $lim_{ t \rightarrow 0^+} \frac{\|A+tB\|-\|A\|}{t} = \tau(A)$ for some $\tau \in \partial_e S_A$. Therefore the evaluation of the derivatives  follows from the description of extreme points and Proposition 2.
\end{proof}
In order to give further examples, in the next set of results the behaviour of norm-weak usc points of the preduality map with respect to $M$-ideals will be studied.
\vskip 1em
For a closed subspace $J \subset X$, we will be using the canonical identification of $J^{\bot\bot}$ with $J^{\ast\ast}$ as well as $(X^\ast/J^\bot)^\ast$. Similarly $(X/J)^{\ast\ast}$ is identified as $X^{\ast\ast}/J^{\bot\bot}$. If $\pi: X \rightarrow X/J$ is the quotient map, then $\pi^{\ast\ast}$ is identified with the quotient map $\pi: X^{\ast\ast} \rightarrow X^{\ast\ast}/J^{\bot\bot}$ without any notational change. We also have $J^{\bot\bot} \cap X = J$.  We implicitly use the canonical contractive projection $Q: X^{\ast\ast\ast} \rightarrow X^{\ast}$ defined by $Q(x^{\ast\ast\ast})|X=x^\ast$ and $ker (Q) = X^\bot$. Similar objects are used with respect to $J$.
\begin{prop}
Let $J \subset X$ be a $M$-ideal and $x \in J$ be a point of norm-weak usc for the preduality map on $J^{\bot\bot}$. Then $x$ is a point of norm-weak usc for the preduality map on $X^{\ast\ast}$. The converse statement is always true.
\end{prop}
\begin{proof}
Let $P: X^\ast \rightarrow X^\ast$ be a linear projection such that $\|x^\ast\|=\|P(x^\ast)\|+\|x^\ast - P(x^\ast)\|$ for all $x^\ast \in X^\ast$ .
For $x^\ast \in X^\ast_1$, if $x^\ast(x) = \|x\|$ then $x^\ast(x) = P(x^\ast)(x)+(x^\ast(x)-P(x^\ast)(x)) = \|x\|$ and $1 = \|x^\ast\|= \|P(x^\ast)\|+\|x^\ast-P(x^\ast)\|$ gives, $\|P(x^\ast)\|= P(x^\ast)(x)$ and $\|x^\ast - P(x^\ast)\|= x^\ast(x)-P(x^\ast)(x)$. This type of relation between the state space vectors will be repeatedly used here.
\vskip 1em
Suppose $J \subset X$ is a $M$-ideal. By duality we have, $X^{\ast\ast}= J^{\bot\bot} \bigoplus_{\infty} (J^\ast)^\bot$ ($\ell^\infty$-direct sum), the conclusion will follow from the more general, easy to see (using the remarks in the preamble, before the Proposition), statement, if $Z = M \bigoplus_1 N$ for closed subspaces $M,N$ of $Z$, then in $Z^\ast = M^\bot \bigoplus_\infty N^\bot = M^\ast \bigoplus_\infty N^\ast$, if $m^\ast \in M^\ast$ is a point of norm-weak usc of the preduality map on $M^\ast$, then $m^\ast$ is also a point of norm-weak usc for the preduality map on $X^\ast$.
\vskip 1em
To see the converse, we write out the state spaces, without using the cumbersome notation of state space with respect to a particular dual.
Let $J \subset X$ be a closed subspace and $x \in X$ be a point of norm-weak usc for the preduality map on $X^{\ast\ast}$. Suppose $x \in J$ and $x^{\ast\ast\ast} \in J^{\ast\ast\ast}_1 $ be such that $x^{\ast\ast\ast}(x) = \|x\|$. By hypothesis, there is a net $\{x^\ast_{\alpha}\}_{\alpha \in \Delta} \subset X^\ast_1$ such that $x^\ast_{\alpha}(x) = \|x\|$ for all $\alpha$ and $x^\ast_{\alpha} \rightarrow x^{\ast\ast\ast}$
in the weak$^\ast$-topology of $X^{\ast}$. Since $x \in J$, clearly the restriction map puts the functionals in the state space with respect to $J^\ast$ and we also have weak$^\ast$ convergence. Hence the conclusion follows.
\end{proof}
\begin{rem}
	We recall that in a $C^\ast$ algebra, $M$-ideals are precisely closed two-sided ideals (see \cite{HWW} Chapter I) and see Chapter V of \cite{HWW} for a description of $M$-ideals in Banach algebras and in ${\mathcal L}(X)$.
\end{rem}
We note that the identity $e$ of a Banach algebra $A$ (and in particular in ${\mathcal L}(X)$) is an extreme point of the unit ball and for the same reason, is an extreme point of $A^{\ast\ast}_1$. If $x_0^\ast \in X^\ast_1$ attains its norm at $x_0\in X_1$ and $x_0^\ast$ is a smooth point, then $x_0 \in \partial_e X^{\ast\ast}_1$.  For a Banach space $X$, an extreme point $x \in X_1$ is said to be a weak$^\ast$-extreme point, if $x$ is also an extreme point of $X^{\ast\ast}_1$. The following proposition illustrates the limitations of some of our assumptions (see Theorem 5). If $J \subset {\mathcal L}(X)$ is a proper $M$-ideal, even if $J$ is not an algebraic ideal, we still get the conclusion $d(I,J) = 1$.
\begin{prop}
Let $x \in X_1$ be a weak$^\ast$-extreme point and let $J \subset X$ be a proper $M$-ideal. Then $d(x,J)=1$. In particular, $\pi(x)$ is a weak$^\ast$-extreme point of $X/J$.
\end{prop}
\begin{proof}
We note that $d(x,J) = d(x,J^{\bot\bot})$. Since $J$ is a $M$-ideal in $X$, we have $X^{\ast\ast}= J^{\bot\bot} \bigoplus_{\infty} (J^\bot)^\ast$. If $P$ is the $M$-projection on $X^{\ast\ast}$ with $ker(P) = J^{\bot\bot}$, then since $x$ is an extreme point of $X^{\ast\ast}_1$, we have $\|P(x)\|=1=\|x-P(x)\|$. Now $d(x,J) = d(x,J^{\bot\bot})= \|P(x)\| =1$. The decomposition above also shows that $\pi(x)$ is an extreme point of $(X/J)^{\ast\ast}= X^{\ast\ast}/J^{\bot\bot}$.
\end{proof}
A particularly interesting situation occurs, when $X$ is a $M$-ideal under the canonical embedding in its bidual $X^{\ast\ast}$. In this case the canonical projection $Q$ on $X^{\ast\ast\ast}$ defined above is a $L$-projection. See Chapter III of \cite{HWW} for examples and Chapter VI when spaces of compact operators ${\mathcal K}(X,Y)$ is a $M$-ideal, when ${\mathcal L}(X^{\ast\ast},Y^{\ast\ast})$ is identified as its bidual. When $X$ is a $M$-ideal in its bidual, $X^\ast$ has the RNP. See \cite{HWW} Theorem III.3.1.
\vskip 1em
It is easy to see that unitaries in a $C^\ast$-algebra $A$, remain as points of norm-weak usc for the preduality map in all duals of higher even order of $A$ (as they are still unitaries in the bigger space). The following phenomenon exhibits weak$^\ast$-dense set of points of norm-weak usc for the preduality map on $X^{\ast\ast}$. We do not know how to determine the largeness of points of norm-weak usc  for the preduality map on $X^{\ast\ast}$ in the category sense, when $X$ is a $M$-ideal in its bidual.
\begin{prop}
	
 Let $X$ be a non-reflexive space that is a $M$-ideal in its bidual.
a) Any $0 \neq x \in X$ is a point of norm-weak usc for the preduality map on $X^{\ast\ast}$ and it continues to be so in all higher duals of even order of $X$.
\vskip 1em
b) If $\tau \in X^{\ast\ast}$ is such that $\tau$ attains its norm on $X^\ast$ and $d(\tau,X)< \|\tau\|$, then $\tau$ is a point of norm-weak usc for the preduality map on $X^{\ast\ast}$. This behaviour continues in all the duals of even order of $X$.
	
\end{prop}
\begin{proof} (a):  Since the projection $Q$ is a $L$-projection, $X^{\ast\ast\ast} = X^\ast \bigoplus_1 X^\bot$, we see that $x^{\ast\ast\ast}$ is the unique extension of its restriction to $X$ and the kernel annihilates $X$.
Hence the conclusion follows.
\vskip 1em
It was shown in \cite{R} that under the hypothesis $X$ continues to be a $M$-ideal of all the higher order even duals of $X$ (all in canonical embeddings). Hence the conclusion follows from  Proposition 6.
\vskip 1em
(b): Since $X^{\ast\ast\ast} = X^\ast \bigoplus_1 X^\bot$ for the canonical projection, the conclusion follows from arguments similar to the ones given during the proof of Theorem 5.
\vskip 1 em
Since we have $X$ is still a $M$-ideal in the fourth dual $X^{(IV)}$ the conclusion follows.
\end{proof}
Next set of results address the quotient space questions and are applicable in particular to the Calkin space ${\mathcal L}(X,Y)/{\mathcal K}(X,Y)$.
\begin{thm}
Let $x$ be a unit vector which is a point of norm-weak usc for the preduality map on $X^{\ast\ast}$. Let $J \subset X$ be a $M$-ideal. Suppose $d(x,J) >0$ (and thus $d(x,J^{\bot\bot}) = d(x,J)>0$) . Then $\pi(x)$ is a point of norm-weak usc for the preduality map on $(X/J)^{\ast\ast}= X^{\ast\ast}/J^{\bot\bot}$.
\end{thm}
\vskip 1em
\begin{proof} Since $X^{\ast} = J^\bot \bigoplus_1 J^\ast$, it is easy to see that $\partial_x = CO(\partial_x \cap J^\ast \cup \partial_x \cap J^\bot)$. Similarly $\partial_{x}' = CO( \partial_x' \cap (J^\ast)^{\bot\bot} \cup \partial_x' \cap J^{\bot\bot\bot})$.
\vskip 1em
By hypothesis, we have, $\partial_x = \{f \in X^{\ast}_1:f(x) = 1\}$ is dense in $\partial_x'= \{\tau \in X^{\ast\ast\ast}_1: \tau(x)= 1\}$, in the weak$^\ast$-topology of $X^{\ast\ast\ast}$. We need to show that $\partial_{\pi(x)} = \{f \in J^\bot_1: f(x) = d(x,J)\} $ is weak$^\ast$-dense in $\partial_{\pi(x)}' = \{\tau \in J^{\bot\bot\bot}_1 : \tau(x) = d(x,J)\}$.
\vskip 1em
Let $P$ be the $L$-projection with range $J^\bot$ and so $P^{\ast\ast}$ is the $L$-projection with range $J^{\bot\bot\bot}$.
\vskip 1em
Let $\tau \in \partial_{\pi(x)}'$. Since $\tau(x) = d(x,J)$, we get, $P^{\ast\ast}(\tau)(x) = \tau(x) = d(x,J)$.
\vskip 1em
Also by hypothesis, there exists a net $\{f_{\alpha}\}_{\alpha \in \Delta} \subset \partial_x$ such that $f_{\alpha} \rightarrow \tau$ in the weak$^\ast$-topology of $X^{\ast\ast\ast}$. Since $P^{\ast\ast}$ is an extension of $P$, we get $P(f_{\alpha}) \rightarrow P^{\ast\ast}(\tau) = \tau$ in the weak$^\ast$-topology of $X^{\ast\ast\ast}$.
Now  using the fact that $P(f_{\alpha})(x) = \|P(f_{\alpha})\|$ for all $\alpha \in \Delta$, we get $\|P(f_{\alpha})\| \rightarrow  d(x,J)>0$.
Therefore the net $$\{\frac{P(f_{\alpha})}{\|P(f_{\alpha})\|d(x,J)}\}_{\alpha \in\Delta} \subset \partial_{\pi(x )}$$ and it converges in the weak$^\ast$-topology of $X^{\ast\ast\ast}$ to
 $\tau$.
 \end{proof}
\vskip 1em
As an application we have the following situation for quotient spaces. We denote by $\pi$ the quotient maps on different spaces, interpreted correctly by the context.
\vskip 1em
\begin{rem}Let $J \subset M \subset X$ be closed subspaces and suppose $M$ is a $M$-ideal in $X$. For $x \in M$, $ x\notin J$, if $\pi(x)$ is a point of norm-weak usc for the preduality map on $M^{\ast\ast}/J^{\bot\bot}$, then it is also a point of norm-weak usc for the preduality map on $X^{\ast\ast}/J^{\bot\bot}$. This follows from the observation, $M/J$ is a $M$-ideal in $X/J$ (see \cite{HWW} Proposition I.1.17). These results are of particular interest in $C^\ast$-algebras with several non-trivial closed two sided ideals. See Theorem 1.3 in \cite{S}. Another application of these ideas comes from Proposition 8.
	
\end{rem}
\begin{rem} Successful application of the norm-weak usc of the preduality map depends on the predual
(see \cite{GI}). In this paper we have assumed the RNP of the space $X$, to ensure uniqueness of $X$ as the predual $X^{\ast}$.
\end{rem}
\vskip 2em
\begin{rem}
  Let $X,Y$ be Banach spaces such that $X^{\ast\ast}$ or $Y^{\ast\ast}$ has the compact metric approximation property. It can be shown that there is a contractive projection $P: {\mathcal L}(X,Y^\ast)^\ast \rightarrow {\mathcal L}(X,Y^\ast)^\ast$ such that $ker P = {\mathcal K}(X,Y^\ast)^\bot$. See the remarks on page 334 of \cite{HWW}. It is easy to verify that this projection is identity on functionals of the form $x^{\ast\ast}\otimes y^{\ast\ast}$. It therefore follows that ${\mathcal K}(X,Y^\ast) \subset {\mathcal L}(X,Y^\ast) \subset {\mathcal K}(X,Y^\ast)^{\ast\ast}$, where the inclusion is the canonical embedding. Suppose the projective tensor product space $X \otimes_{\pi} Y$ has the RNP (hence it is the unique predual of ${\mathcal L}(X,Y^\ast))$. Let $A \in {\mathcal K}(X,Y^\ast)$ be point of norm-weak usc in the bidual, we do not know if it is also a point of norm-weak usc in the intermediate space${\mathcal L}(X,Y^\ast)$?
  A difficulty being, that ${\mathcal L}(X,Y^\ast)$ need not be a weak$^\ast$-closed subspace of ${\mathcal K}(X,Y)^{\ast\ast}$.
\end{rem}
We recall from \cite{FS} that a smooth point $x \in X$ is said to be a very smooth point, if $x$ is also a smooth point of $X^{\ast\ast}$. See \cite{R1} for an analysis of very smooth points in spaces of operators.
\begin{prop}
	Let $X,Y$ be Banach spaces such that $X^{\ast\ast}$ or $Y^{\ast\ast}$ has the compact metric approximation property. Suppose $A \in {\mathcal K}(X,Y^\ast)$ is a very smooth point. Then for any $B \in {\mathcal L}(X,Y^\ast)$, there exists $x^{\ast\ast} \in \partial_e X^{\ast\ast}_1$ and $y^{\ast\ast} \in \partial_e Y^{\ast\ast}_1$ such that $x^{\ast\ast}(A^\ast(y^{\ast\ast}))=\|A\|$ and $$lim_{ t \rightarrow 0^+} \frac{\|A+tB\|-\|A\|}{t} = x^{\ast\ast}(B^\ast(y^{\ast\ast})).$$
\end{prop}
\begin{proof} By Remark 13, the hypothesis implies ${\mathcal K}(X,Y^\ast) \subset {\mathcal L}(X,Y^\ast) \subset {\mathcal K}(X,Y^\ast)^{\ast\ast}$ in the canonical embedding. Since $A$ is a very smooth point, we get that $A$ is a smooth point of ${\mathcal L}(X,Y^\ast)$. Since $A$ is also a smooth point of ${\mathcal K}(X,Y^\ast)$ by the structure of the extreme points, we get $S_A = \{x^{\ast\ast}_0 \otimes y^{\ast\ast}_0\}$ for some $x^{\ast\ast} \in \partial_e X^{\ast\ast}_1$ and $y^{\ast\ast} \in \partial_e Y^{\ast\ast}_1$. Hence the conclusion follows.
	
\end{proof}
We conclude the paper by studying a notion where subdifferential limits exists uniformly. We recall from \cite{FP} that a unit vector $x$ is said to be a point of strong subdifferentiability (SSD for short) if $lim_{t \rightarrow 0^+}\frac{\|x+ty\|-\|x\|}{t}$ exists uniformly for $y \in X_1$. To connect with the preceding analysis, we note from \cite{FP} that,
$x \in X$ is a SSD point if and only if it is a SSD point in $X^{\ast\ast}$ and a SSD point is a point of norm-norm continuity of the duality map (see \cite{FP} Theorem 1.2).
Thus it is easy to see that a unit vector, $x$ is a SSD point, if and only if for $\epsilon>0$ there is a $\delta>0$ such that whenever $\|z\|=1$, $\|z-x\|< \delta$ implies $S_z \subset S_x+ \epsilon X^\ast_1$. For a later use, we note that $S_x + \epsilon X^\ast_1$ is a weak$^\ast$-compact convex set, so the verify the inclusion one only needs $\partial_e S_z \subset S_x+ \epsilon X^\ast_1$. It is easy to see that as in Remark 13, assuming compact metric approximation property on $X^\ast$ or $Y$, we have that any SSD point of ${\mathcal K}(X,Y)$ is a SSD point of ${\mathcal L}(X,Y)$.
\vskip 1em
The next proposition we give a different proof to show that strong subdifferentiability of the norm passes through $M$-ideals. It was shown in \cite{CP} (Corollary 1) that subdifferentiability passes from a $C^\ast$-subalgebra to a $C^\ast$-algebra. As $M$-ideals in $C^\ast$-algebras are precisely closed two-sided ideals (see Theorem V.4.4 in \cite{HWW}), we get a new geometric proof of Corollary 1, in the case of ideals, using the state-space method. It can also be deduced from Proposition 2.1 in \cite{FP}. We do not know how to compare subdifferential limits, in the case of $C^\ast$-subalgebras?
\begin{prop}
  Let $J \subset X$ be a $M$-ideal. If a unit vector $j \in J$ is a SSD point of $J$, then it is a SSD point of $X$.
\end{prop}
\begin{proof}
  Since $J \subset X$ is a $M$-ideal, we have, $X^{\ast\ast} = J^{\bot\bot} \bigoplus_{\infty} (J^\ast)^\bot$. If $j$ is a SSD point of $J$, we have from \cite{FP} that it is a SSD point of $J^{\bot\bot}=J^{\ast\ast}$. Thus it is enough to show that this property passes from $M$-summands. Then $j$ will be a SSD point of $X^{\ast\ast}$ and hence in $X$.
  \vskip 1em
  Assume without loss of generality, $X = J \bigoplus_{\infty} J'$. We first note that $S_j$ is the same set in both $J$ and $X$. We will use the state space definition of a SSD point. Let $\epsilon >0$ and let $\delta$ be as in the definition of $j$ being a SSD point in $J$. Let $\delta'= min \{1,\delta\}$. Let $z \in X$, $\|z\|=1 $, $z =j_1+j_2$ for $j_1 \in J$, $j_2 \in J'$ , so that $max\{\|j_1\|,\|j_2\|\}=1$  and $\|z-j\|= max\{\|j_1-j\|,\|j_2\|\}< \delta'$. As before we only need to consider extreme points of the state space $S_z$. Let $\phi \in \partial_e S_z$. If $\phi \in J^\ast = (J')^\bot$, then $\phi(z)= \phi(j_1)=1$, hence $ \|j_1\|=1$. As $\|j_1-j\|< \delta$ by hypothesis, $ \phi \in S_j + \epsilon J^\ast_1 \subset S_j + \epsilon X^\ast_1$.
  \vskip 1em
  Suppose $\phi \in J^\bot$. Then $\phi(z)= \phi(j_2)=1$ and hence $\|j_2\|=1$. This contradicts, $\|j_2\|< \delta'$. Thus $j$ is a SSD point of $X$.
\end{proof}

\end{document}